\documentclass[12pt]{article} 

\usepackage[margin=3cm]{geometry}
\usepackage[greek, english]{babel}
\usepackage[pagebackref]{hyperref}
\usepackage{amsmath}
\usepackage{amsthm}
\usepackage{amssymb}
\usepackage{graphics}
\usepackage{ucs}
\usepackage[all]{xy}
\usepackage{enumerate}
\usepackage[utf8x]{inputenc}
\usepackage[T1]{fontenc}
\usepackage{xcolor}
\usepackage{tikz}
 \usepackage{setspace}
\DeclareMathOperator{\supp}{\mathrm{supp}}

\newcommand{\R}{\mathbb R}
\newcommand{\N}{\mathbb N}

\newcommand{\Z}{\mathbb Z}

\def\pmp{p.m.p.}

\newcommand{\URS}{\mathrm{URS}}
\newcommand{\IRS}{\mathrm{IRS}}
\newcommand{\Stab}{\mathrm{Stab}}
\newcommand{\Fix}{\mathrm{Fix}}
\newcommand{\Sub}{\mathrm{Sub}}

\newcommand{\onto}{\twoheadrightarrow}

\newtheorem{thm}{Theorem}[section]
\newtheorem{cor}[thm]{Corollary}

\newtheorem{lem}[thm]{Lemma}
\newtheorem{prop}[thm]{Proposition}

\theoremstyle{definition}

\newtheorem{qu}[thm]{Question}
\newtheorem*{qu*}{Question}
\newtheorem*{mainqu*}{Main Question}

\newtheorem{df}[thm]{Definition}
\newtheorem{rmq}[thm]{Remark}

\makeatletter
\newcommand{\oset}[3][0ex]{%
 \mathrel{\mathop{#3}\limits_{
 \vbox to#1{\kern-5\ex@
 \hbox{$\scriptstyle#2$}\vss}}}}
\makeatother

\title{Continuum of allosteric actions for non-amenable surface groups}
\author{Matthieu Joseph}

\date{}
\begin{document}

	\maketitle


\newcommand{\Addresses}{{
 \bigskip
 \footnotesize\noindent M. Joseph, \textsc{Université de Lyon, ENS de Lyon, Unité de Mathématiques Pures et Appliquées, 46,
allée d’Italie 69364 Lyon Cedex 07, FRANCE}\par\nopagebreak\noindent
 \textit{E-mail address: }\texttt{matthieu.joseph@ens-lyon.fr}
}}

	\begin{abstract}
Let $\Sigma$ be a closed surface other than the sphere, the torus, the projective plane or the Klein bottle. We construct a continuum of \pmp{} ergodic minimal profinite actions for the fundamental group of $\Sigma$, that are topologically free but not essentially free, a property that we call allostery. Moreover, the IRS's we obtain are pairwise distincts. \end{abstract}


\section{Introduction}

Let $\Gamma$ be a countable discrete group. Let $\alpha$ be a minimal action of $\Gamma$ on a compact Hausdorff space $C$. The action $\alpha$ is \emph{\textbf{topologically free}} if for every non-trivial element $\gamma\in\Gamma$, the set
$\{x\in C\mid \alpha(\gamma)x=x\}$ has empty interior. This notion of freeness can be characterized by the triviality of the URS associated with the action $\alpha$ as follows. Let $\Sub(\Gamma)$ be the space of subgroups of $\Gamma$, and let $\Stab_\alpha : C\to\Sub(\Gamma)$ be the Borel map defined by
\[\Stab_\alpha(x):=\{\gamma\in\Gamma\mid \alpha(\gamma)x=x\}.\]
Here $\Sub(\Gamma)$ is equipped with the topology of pointwise convergence which turns it into a compact totally disconnected topological space on which $\Gamma$ acts continuously by conjugation. Glasner and Weiss proved in \cite{GlasnerWeiss} that there exists a unique closed, $\Gamma$-invariant, minimal subset in the closure of $\{\Stab_\alpha(x)\mid x\in C\}$, called the \emph{\textbf{stabilizer Uniformly Recurrent Subgroup}}, stabilizer URS for short, associated with the minimal action $\alpha$, that we denote by $\URS(\alpha)$. The stabilizer URS is trivial if it is equal to $\{\{1\}\}$. One of the feature of the stabilizer URS associated with a minimal action $\alpha$ is that its triviality is equivalent to the topological freeness of $\alpha$, see Lemma \ref{lem.URStrivial}.

Let $(X,\mu)$ be a standard probability measure space, and let $\beta$ be a probability measure preserving (hereafter \pmp{}) action of a countable group $\Gamma$ on $(X,\mu)$. The action $\beta$ is \textbf{\emph{essentially free}} if for every non-trivial $\gamma\in\Gamma$, the set $\{x\in X\mid \alpha(\gamma)x=x\}$ is $\mu$-negligible. The measurable counterpart of the stabilizer URS is the \emph{\textbf{stabilizer Invariant Random Subgroup}}, stabilizer IRS for short, associated with $\beta$. It is defined as the $\Gamma$-invariant Borel probability measure $(\Stab_\beta)_*\mu$ on $\Sub(\Gamma)$, and is denoted by $\IRS(\beta)$. A stabilizer IRS is the prototype of an \emph{\textbf{IRS}}, which is a Borel probability measure on $\Sub(\Gamma)$ that is invariant under the conjugation action of $\Gamma$. The \emph{\textbf{trivial IRS}} is the Dirac measure at the trivial subgroup.
Observe that $\IRS(\beta)$ is trivial if and only if $\beta$ is essentially free. Abért, Glasner and Virág proved that every IRS is in fact a stabilizer IRS for some \pmp{} action, see \cite{AbertGlasnerVirag}.

An \emph{\textbf{ergodic minimal action}} $\Gamma\curvearrowright (C,\mu)$ is a minimal action of $\Gamma$ on a compact Hausdorff space $C$ together with a $\Gamma$-invariant ergodic Borel probability measure $\mu$. Thus an ergodic minimal action has both a stabilizer URS and a stabilizer IRS. It is a classical result that the essential freeness of an ergodic minimal action implies its topological freeness, see Lemma \ref{lem.essfreeimpliestopfree}. In other words, if the stabilizer IRS of an ergodic minimal action is trivial, then its stabilizer URS is trivial. The present article provides new counterexamples in the study of the converse.

\begin{df} An ergodic minimal action is \emph{\textbf{allosteric}}\footnote{\textgreek{ἄλλος}: other, \textgreek{στερεός}: fix, firm, solid, rigid} if it is topologically free but not essentially free. A group is allosteric if it admits an allosteric action.
\end{df}


\begin{mainqu*}
What is the class of allosteric groups?
\end{mainqu*}

First, let us discuss examples of groups that don't belong to this class. It is the case for groups whose ergodic IRS's are all atomic, i.e., equal to the uniform measure on the set of conjugates of a finite index subgroup. Indeed, we prove in Proposition \ref{prop.atomless} that the IRS of an ergodic minimal action which is topologically free is either trivial, or has no atoms. Thus, if $\Sub(\Gamma)$ is countable, then $\Gamma$ is not allosteric, see Corollary \ref{cor.SubCountable}. Examples of groups with only countably many subgroups are: finitely generated nilpotent groups, more generally polycyclic groups, extensions of Noetherian groups by groups with only countably many subgroups (e.g. solvable Baumslag-Solitar groups $BS(1,n)$), see \cite{BeckerLubotzkyThom}, or Tarski monsters.

There are also groups whose ergodic IRS's are all atomic for other reasons. For instance, this is the case for lattices in simple higher rank Lie groups \cite{StuckZimmer}, commutator subgroups of either a Higman-Thompson group or the full group of an irreducible shift of finite type \cite{DudkoMedynets}, projective special linear group $\mathrm{PSL}_n(k)$ over an infinite countable field $k$ \cite{PetersonThom}. See also \cite{Creutz}, \cite{CreutzPeterson} or \cite{Bekka} for other examples of groups with few ergodic IRS's. Thus, none of these groups are allosteric, because of their lack of IRS's.

More surprisingly, there exists non-allosteric groups with plenty of ergodic IRS's, such as countable abelian groups which admit uncountably many subgroups. Indeed, if $\Gamma$ is such a group, then any Borel probability measure on $\Sub(\Gamma)$ is an IRS, but $\Gamma$ is not allosteric since any minimal $\Gamma$-action which is topologically free is actually essentially free for any invariant measure, see Remark \ref{rmq.totipotent}.Another example is given by the group $\text{FSym}(\N)$ of finitely supported permutations on $\N$, as well as its alternating subgroup $\text{Alt}(\N)$. They both admit a lot of ergodic IRS's, see \cite{Vershik} and \cite{ThomasTuckerDrob}. However, an argument similar to that of Lemma $10.4$ in \cite{ThomasTuckerDrob} implies that neither $\text{FSym}(\N)$ nor $\text{Alt}(\N)$ is allosteric.



Let us now discuss examples of allosteric groups. Bergeron and Gaboriau proved in \cite{BergeronGaboriau} that if $\Gamma$ is non-amenable and isomorphic to a free product of two non-trivial residually finite groups, then $\Gamma$ is allosteric. We refer to Remark \ref{rmq.BG} for a more precise statement of their results. In \cite{AbertElek}, Abért and Elek independently proved that finitely generated non-abelian free groups are allosteric, and in \cite{AbertElek2}, they proved that the free product of four copies of $\Z/2\Z$ admits an allosteric action whose orbit equivalence relation is measure hyperfinite. In all \cite{BergeronGaboriau}, \cite{AbertElek} and \cite{AbertElek2}, the allosteric actions obtained are in fact profinite, see Section \ref{sec.profinite} for a definition. These were the first known examples answering a question of Grigorchuk, Nekrashevich and Sushchanskii in \cite[Problem 7.3.3]{GrigorchukNekrashevichSushchanskii} about the existence of profinite allosteric actions.

The main result of this article is to prove that non-amenable \emph{\textbf{surface groups}}, that is fundamental groups of closed surfaces other than the sphere, the torus, the projective plane or the Klein bottle, are allosteric. More precisely, we prove the following result.

\begin{thm}\label{thm.main} Any non-amenable surface group admits a continuum of profinite allosteric actions that are pairwise topologically and measurably non-isomorphic.\end{thm}

Moreover, we prove that the IRS's given by the non-isomorphic allosteric actions that we construct are pairwise distinct. We refer to Theorems \ref{thm.mainraffinerorientable} and \ref{thm.mainraffinernonorientable} for a precise statement of our results. Let us mention that surface groups are known to have a large "zoo" of IRS's. For instance, Bowen, Grigorchuk and Kravchenko proved in \cite{BowenGrigorchukKravchenko} that any non elementary Gromov hyperbolic group admits a continuum of IRS's which are weakly mixing when considered as dynamical systems on $\Sub(\Gamma)$. In an upcoming work (personal communication), Carderi, Le Maître and Gaboriau prove that non-amenable surface groups admit a continuum of IRS's whose support coincides with the \emph{\textbf{perfect kernel}} of $\Gamma$, i.e., the largest closed subset without isolated points in $\Sub(\Gamma)$. However, our IRS's are drastically different from the latter ones: we show that they are not weakly mixing, and that their support is strictly smaller than the perfect kernel, see Remarks \ref{rmq.totipotent} and \ref{rmq.weaklymixing}.

We develop in Section \ref{sec.pre} the preliminary results needed about profinite actions and allosteric actions. In particular, we prove that allostery is invariant under commensurability. In order to build ergodic profinite allosteric actions of non-amenable surface groups, we rely on a residual property of non-amenable surface groups in order to prove in Section \ref{sec.finiteindex} that they admit special kinds of finite index subgroups. The proof of Theorem \ref{thm.main} is completed in Section \ref{sec.proof}.

\paragraph{Acknowledgments.} I would like to thank A. Le Boudec for various conversations related to this article, as well as T. Nagnibeda for letting me know that Vershik's work implies that the group of finitely supported permutations of the integers is not allosteric. I wish to thank particularly D. Gaboriau for his constant encouragement and support, as well as for his numerous remarks on this article.


\section{Preliminaries}\label{sec.pre}

\subsection{Topological dynamic and URS/IRS}

Let $C$ be a compact Hausdorff space, and let $\alpha$ be an action by homeomorphisms of a countable discrete group $\Gamma$ on $C$. The action $\alpha$ is \emph{\textbf{minimal}} if the orbit of every $x\in C$ is dense. Recall that $\alpha$ is topologically free if for every non-trivial element $\gamma\in\Gamma$, the closed set
\[\Fix_\alpha(\gamma):=\{x\in C\mid \alpha(\gamma)x=x\}\]
has empty interior. Since $C$ is a Baire space, this is equivalent to saying that the set $\{x\in C\mid \Stab_\alpha(x)\neq \{1\}\}$ is meager, i.e., a countable union of nowhere dense sets.

The set $\Sub(\Gamma)$ of subgroups of $\Gamma$ naturally identifies with a subset of $\{0,1\}^\Gamma$. It is closed for the product topology. Thus the induced topology on $\Sub(\Gamma)$ turns it into a compact totally disconnected space, on which $\Gamma$ acts continuously by conjugation. A \emph{\textbf{URS}} of $\Gamma$ is a closed minimal $\Gamma$-invariant subset of $\Sub(\Gamma)$. The \emph{\textbf{trivial URS}} is the URS that only contains the trivial subgroup. Recall that the stabilizer URS of a minimal action $\alpha$ of $\Gamma$ on $C$ is the unique closed, $\Gamma$-invariant minimal subset in the closure of $\{\Stab_\alpha(x)\mid x\in C\}$. If $C_0\subset C$ denotes the locus of continuity of $\Stab_\alpha : C\to \Sub(\Gamma)$, then one can prove that $\URS(\alpha)$ is equal to the closure of the set $\{\Stab_\alpha(x)\mid x\in C_0\}$, see \cite{GlasnerWeiss}. 

A proof of the following classical result can be found in \cite[Prop. 2.7]{LeBoudecMatteBon}.

\begin{lem}\label{lem.URStrivial}Let $\alpha$ be a minimal $\Gamma$-action on a compact Hausdorff space $C$. Then $\alpha$ is topologically free if and only if its stabilizer URS is trivial, if and only if there exists $x\in C$ such that $\Stab_\alpha(x)$ is trivial.
\end{lem}



The following lemma clarifies the relation between the stabilizer URS and the stabilizer IRS. Recall that the support of a Borel probability measure is the intersection of all closed subsets of full measure. 

\begin{lem}\label{lem.essfreeimpliestopfree} 
Let $\alpha$ be a minimal $\Gamma$-action on a compact Hausdorff space $C$ and $\mu$ be a $\Gamma$-invariant Borel probability measure on $C$. Then $\URS(\alpha)$ is contained in the support of $\IRS(\alpha)$. In particular, if $\IRS(\alpha)$ is trivial, then $\URS(\alpha)$ is trivial.
\end{lem}

\begin{proof} Let $F$ be a closed subset of $\Sub(\Gamma)$ such that $\mu(\Stab_{\alpha}^{-1}(F))=1$. By minimality of $\alpha$, every non-empty open subset $U$ of $C$ satisfies $\mu(U)>0$. Thus, $\Stab_\alpha^{-1}(F)$ is dense in $C$. Let $x\in C$ be a continuity point of $\Stab_\alpha$. Let $(x_n)_{n\geq 0}$ be a sequence of elements in $\Stab_\alpha^{-1}(F)$ that converges to $x$. Then $\Stab_\alpha(x)\in F$, and we thus obtain that $\URS(\alpha)\subset F$. By definition of the support of $\IRS(\alpha)$, this implies that $\URS(\alpha)\subset\supp(\IRS(\alpha))$. 
\end{proof}

The following proposition gives a partial converse to Lemma \ref{lem.essfreeimpliestopfree}.

\begin{prop}\label{prop.atomless} Let $\alpha$ be a minimal $\Gamma$-action on a compact Hausdorff space $C$, and $\mu$ be a $\Gamma$-invariant Borel probability measure on $C$. If $\URS(\alpha)$ is trivial, then $\IRS(\alpha)$ is either trivial or atomless.\end{prop}

\begin{proof} Assume that $\IRS(\alpha)$ has a non-trivial atom $\{\Lambda\}$. By invariance, the atoms $\{\gamma\Lambda\gamma^{-1}\}$ have equal measure for all $\gamma\in\Gamma$. Thus, $\Lambda$ has only finitely many conjugates. Thus, the closure in $\Sub(\Gamma)$ of the set $\{\Stab_\alpha(x)\mid x\in C\}$ contains the finite set $\{\gamma\Lambda\gamma^{-1}\mid\gamma\in\Gamma\}$, which is closed, $\Gamma$-invariant and minimal. Thus, $\URS(\alpha)$ is non-trivial.
\end{proof}

This last result implies that the converse of Lemma \ref{lem.essfreeimpliestopfree} is actually true for groups admitting only countably many subgroups.
\begin{cor}\label{cor.SubCountable} Let $\alpha$ be a minimal $\Gamma$-action on a compact Hausdorff space and $\mu$ a $\Gamma$-invariant Borel probability measure on $C$. If $\Sub(\Gamma)$ is countable, then $\IRS(\alpha)$ is trivial iff $\URS(\alpha)$ is trivial.
\end{cor}

Thus, groups $\Gamma$ such that $\Sub(\Gamma)$ is countable are not allosteric.


\subsection{Profinite actions and their URS/IRS}\label{sec.profinite}

Let $\Gamma$ be a countable group. For every $n\geq 0$, let $\alpha_n$ be a $\Gamma$-action on a finite set $X_n$, and assume that for every $n\geq 0$, $\alpha_{n}$ is a quotient of $\alpha_{n+1}$, i.e., there exists a $\Gamma$-equivariant onto map $q_n : X_{n+1}\twoheadrightarrow X_n$. The inverse limit of the finite spaces $X_n$ is the space
\[\varprojlim X_n:=\left\{(x_n)_{n\geq 0}\in \prod_{n\geq 0}X_n\mid \forall n\geq 0, q_n(x_{n+1})=x_n\right\}.\]
This space is closed, thus compact, and totally disconnected in the product topology. Let $\alpha$ be the $\Gamma$-action by homeomorphisms on $\varprojlim X_n$ defined by
\[\alpha(\gamma)(x_n)_{n\geq 0}:=(\alpha_n(\gamma)x_n)_{n\geq 0}.\]
If each $X_n$ is endowed with a $\Gamma$-invariant probability measure $\mu_n$, we let $\mu$ be the unique Borel probability measure on $\varprojlim X_n$ that projects onto $\mu_k$ via the canonical projection $\pi_k : \varprojlim X_n\to X_k$, for every $k\geq 0$. The $\Gamma$-action $\alpha$ preserves $\mu$, and is called the \emph{\textbf{inverse limit}} of the \pmp{} $\Gamma$-actions $\alpha_n$. A \pmp{} action of $\Gamma$ is \emph{\textbf{profinite}} if it is measurably isomorphic to an inverse limit of \pmp{} $\Gamma$-actions on finite sets. A proof of the following Lemma can be found in \cite[Prop. 4.1]{Grigorchuk}.
\begin{lem} The following are equivalent:
\begin{enumerate}\label{lem.ergodicminimal}
\item For every $n\geq 0, \alpha_n$ is transitive, and $\mu_n$ is the uniform measure on $X_n$.
\item The action $\alpha$ is minimal.
\item The action $\alpha$ is $\mu$-ergodic.
\item The action $\alpha$ is uniquely ergodic, i.e., $\mu$ is the unique $\Gamma$-invariant Borel probability measure on $\varprojlim X_n$.
\end{enumerate}
\end{lem}


With the above notations, the following lemma is useful to compute the measure of a closed subset in an inverse limit (here, no group action is involved).

\begin{lem}\label{lem.closed} Let $A$ be a closed subset of $\varprojlim X_n$. Then $A=\bigcap_{n\geq 0}\pi_n^{-1}(\pi_n(A))$. Thus
\[\mu(A)=\underset{n\to +\infty}{\lim}\mu_n(\pi_n(A)).\]
\end{lem}
\begin{proof} First, $A$ is contained in $\cap_{n\geq 0}\pi_n^{-1}(\pi_n(A))$ since it is contained in each $\pi_n^{-1}(\pi_n(A))$. Conversely, let $x$ be in $\cap_{n\geq 0}\pi_n^{-1}(\pi_n(A))$. For every $n\geq 0$, there exists $y_n\in A$ such that $\pi_n(x)=\pi_n(y_n)$. By compactness of $A$, let $y\in A$ be a limit of some subsequence of $(y_n)_{n\geq 0}$. By definition of the product topology, for every $n\geq 0, \pi_n(x)=\pi_n(y)$, thus $x=y$ and $x$ belongs to $A$.
\end{proof}
Let $(\Gamma_n)_{n\geq 0}$ be a \emph{\textbf{chain}} in $\Gamma$, that is an infinite decreasing sequence $\Gamma=\Gamma_0\geq \Gamma_1\geq \dots$ of finite index subgroups. If $X_n=\Gamma/\Gamma_n$ and $\mu_n$ is the uniform probability measure on $X_n$, then we get a profinite action that is ergodic by Lemma \ref{lem.ergodicminimal}. Conversely, any ergodic (equivalently minimal) profinite $\Gamma$-action $\Gamma\curvearrowright\varprojlim X_n$ is measurably isomorphic to a profinite action of the form $\Gamma\curvearrowright\varprojlim \Gamma/\Gamma_n$ for some chain $(\Gamma_n)_{n\geq 0}$, by fixing a point $x\in\varprojlim X_n$, and letting $\Gamma_n$ be the stabilizer of $\pi_n(x)\in X_n$.


\begin{lem}\label{lem.URSprofinitetrivial} Let $(\Gamma_n)_{n\geq 0}$ be a chain in $\Gamma$, and let $\alpha$ be the corresponding ergodic profinite $\Gamma$-action. Then $\URS(\alpha)$ is trivial if and only if there exists a sequence $(\gamma_n)_{n\geq 0}$ of elements in $\Gamma$ such that
\[\bigcap_{n\geq 0}\gamma_n\Gamma_n\gamma_n^{-1}=\{1\}.\]
\end{lem}

 \begin{proof}
For all $x\in\varprojlim \Gamma/\Gamma_n$, if $x=(\gamma_n\Gamma_n)_{n\geq 0}$, then
 \[\Stab_\alpha(x)=\bigcap_{n\geq 0}\gamma_n\Gamma_n\gamma_n^{-1}.\]
Thus, the result is a direct consequence of Lemma \ref{lem.URStrivial}.
 \end{proof}

 \begin{prop}\label{prop.mixing}
 Let $(\Gamma_n)_{n\geq 0}$ be a chain in $\Gamma$, and let $\alpha$ be the corresponding ergodic profinite $\Gamma$-action. If $\URS(\alpha)$ is trivial, then either $\IRS(\alpha)$ is trivial, or there exists a finite index subgroup $\Lambda\leq\Gamma$ such that the \pmp{} $\Lambda$-action by conjugation on $(\Sub(\Gamma),\IRS(\alpha))$ is not ergodic.
 \end{prop}

\begin{proof} Assume that the \pmp{} $\Gamma$-action by conjugation on $(\Sub(\Gamma),\IRS(\alpha))$ remains ergodic under any finite index subgroup of $\Gamma$.
Since $\URS(\alpha)$ is trivial, there exists by Lemma \ref{lem.URSprofinitetrivial} a sequence $(\gamma_n)_{n\geq 0}$ of elements in $\Gamma$ such that
\[\bigcap_{n\geq 0}\gamma_n\Gamma_n\gamma_n^{-1}=\{1\}.\]
For every $k\geq 0$, if $\pi_k : \varprojlim \Gamma/\Gamma_n\to\Gamma/\Gamma_k$ denotes the projection onto the $k^{\text{th}}$ coordinate, then the set
\[\{\Stab_\alpha(x)\mid x\in\varprojlim\Gamma/\Gamma_n, \pi_k(x)=\gamma_k\Gamma_k\}\subset\Sub(\Gamma)\]
has positive measure for $\IRS(\alpha)$, is contained in $\Sub(\gamma_k\Gamma_k\gamma_k^{-1})$ and is invariant under the finite index subgroup $\Stab_{\alpha_k}(\gamma_k\Gamma_k)=\gamma_k\Gamma_k\gamma_k^{-1} $. By ergodicity, it is a full measure set. Thus, for a.e. $x\in\varprojlim \Gamma/\Gamma_n$, $\Stab_\alpha(x)$ is a subgroup of $\gamma_k\Gamma_k\gamma_k^{-1}$. Since this is true for every $k\geq 0$, we conclude that $\IRS(\alpha)$ is trivial.
\end{proof}

\subsection{Allostery and commensurability}

Two groups $\Gamma_1$ and $\Gamma_2$ are \emph{\textbf{commensurable}} if there exists finite index subgroups $\Lambda_1\leq \Gamma_1$ and $\Lambda_2\leq \Gamma_2$ such that $\Lambda_1$ is isomorphic to $\Lambda_2$. In this section, we prove the following result.

\begin{thm}\label{thm.allosterycommensurability} Allostery is invariant under commensurability.
\end{thm}

We prove Theorem \ref{thm.allosterycommensurability} in two steps, by showing that allostery is inherited by finite index overgroups in Proposition \ref{prop.allostericsousgroupe} and by finite index subgroups in Proposition \ref{prop.allostericsurgroupe}. Let $\Gamma$ be a countable group and $\Lambda\leq\Gamma$ a finite index subgroup. Let $\alpha : \Lambda\curvearrowright (C,\mu)$ be an action by homeomorphisms on a compact Hausdorff space $C$ with a $\Lambda$-invariant Borel probability measure $\mu$ on $C$. The group $\Gamma$ acts on $X\times\Gamma$ trivially on the first factor and by left multiplication on the second factor. This action projects onto a $\Gamma$-action by homeomorphisms on the quotient of $X\times \Gamma$ by the $\Lambda$-action $\lambda\cdot (x,\gamma)=(\alpha(\lambda)x,\gamma\lambda)$, and the product of $\mu$ with the counting measure projects onto a $\Gamma$-invariant Borel probability measure. This action is the $\Gamma$-action \emph{\textbf{induced}} by $\alpha$.

\begin{prop}\label{prop.allostericsousgroupe}
Let $\Gamma$ be a countable group and $\Lambda\leq\Gamma$ a finite index subgroup. Then the $\Gamma$-action induced by any allosteric $\Lambda$-action is allosteric.\end{prop}

\begin{proof} Let $\alpha : \Lambda\curvearrowright (C,\mu)$ be an allosteric action. It is an exercise to prove that the $\Gamma$-action $\beta$ induced by $\Lambda$ is ergodic and minimal. Moreover, $\IRS(\beta)$ is non-trivial since the restriction of $\beta$ to $\Lambda$ is not essentially free. Finally, $\URS(\alpha)$ is trivial, thus there exists by Lemma \ref{lem.URStrivial} a point $x\in C$ such that $\Stab_\alpha(x)=\{1\}$. Let $y$ be the projection of $(x,1)$ onto the quotient $(C\times\Gamma)/\Lambda$, then $\Stab_\beta(y)=\{1\}$. Since $\beta$ is minimal, this implies by Lemma \ref{lem.URStrivial} that $\URS(\beta)$ is trivial. Thus $\beta$ is allosteric.
\end{proof}

\begin{prop}\label{prop.allostericsurgroupe} Any finite index subgroup of an allosteric group is allosteric.
\end{prop}

\begin{proof} Let $\Lambda\leq\Gamma$ be a finite index subgroup. We recall the following two facts. If $\Gamma\curvearrowright (X,\mu)$ is an ergodic action, then any $\Lambda$-invariant measurable set $A\subset X$ of positive measure satisfies $\mu(A)\geq 1/[\Gamma:\Lambda]$. Moreover, for any $\Lambda$-invariant measurable set $B\subset X$ of positive measure, there exists a $\Lambda$-invariant measurable set $A\subset B$ of positive measure on which $\Lambda$ acts ergodically.

Let $\Gamma$ be an allosteric group, and let $\Lambda\leq \Gamma$ be a finite index subgroup. Let $N$ be the normal core of $\Lambda$ (the intersection of the conjugates of $\Lambda$). It is a finite index normal subgroup of $\Gamma$ which is contained in $\Lambda$. We will prove that $N$ is allosteric. Proposition \ref{prop.allostericsousgroupe} will then imply that $\Lambda$ is allosteric. We let $d=[\Gamma:N]$ and we fix $\gamma_1,\dots,\gamma_d\in\Gamma$ a coset representative system for $N$ in $\Gamma$. Let $\Gamma\curvearrowright^\alpha (C,\mu)$ be an allosteric action. For all $x\in C$, we define $\mathcal{O}_N(x)=\overline{\{\alpha(\gamma)x\mid \gamma\in N\}}$. This is a closed, $N$-invariant subset of $C$. By minimality of $\alpha$, for all $x\in C$,\[X=\bigcup_{i=1}^d\mathcal{O}_N(\alpha(\gamma_i)x).\]
Moreover, since $N$ is normal in $\Gamma$, for all $x\in C$ and $\gamma\in\Gamma$, we have $\mathcal{O}_N(\alpha(\gamma)x)=\alpha(\gamma)\mathcal{O}_N(x)$. This implies that $\mu(\mathcal{O}_N(\alpha(\gamma)x))=\mu(\mathcal{O}_N(x))$ and that $\mu(\mathcal{O}_N(x))>0$. Let $y$ be a point in some closed, $N$-invariant and $N$-minimal set. Then $N\curvearrowright\mathcal{O}_N(y)$ is minimal. Let $A\subset\mathcal{O}_N(y)$ be a $N$-invariant measurable set of positive measure on which $N$ acts ergodically. Let $\mu_A$ be the Borel probability measure on $A$ induced by $\mu$. Then $N\curvearrowright(\mathcal{O}_N(y),\mu_A)$ is an ergodic minimal action, which is still topologically free. Let us prove that it is not essentially free. Since $\alpha$ is allosteric, $\IRS(\alpha)$ is atomless, see Proposition \ref{prop.atomless}. Thus, for $\mu$-a.e. $x\in C$, $\Stab_{\alpha}(x)$ is infinite. Since $N$ has finite index in $\Gamma$, this implies that for $\mu$-a.e. $x\in C$, $\Stab_\alpha(x)\cap N$ is infinite. Thus $N\curvearrowright (\mathcal{O}_N(y),\mu_A)$ is not essentially free, and thus is allosteric.
\end{proof}

\begin{rmq}\label{rmq.BG} It is proved in \cite[Théorème $4.1$]{BergeronGaboriau} that if $\Gamma$ is isomorphic to a free product of two \emph{infinite} residually finite groups, then $\Gamma$ admits a continuum of profinite allosteric actions. Let $\Gamma'$ be a non-amenable group which is isomorphic to a free product of two non-trivial residually finite groups. Then Kurosh's theorem \cite[Section $5.5$]{Serre} implies that $\Gamma'$ admits a finite index subgroup $\Gamma$ isomorphic to a free product of finitely many (and at least two) residually finite infinite groups. Proposition \ref{prop.allostericsousgroupe} then implies that $\Gamma'$ is allosteric.
\end{rmq}

\section{Finite index subgroups of surface groups}\label{sec.finiteindex}
\subsection{Residual properties of surface groups}

A surface group is the fundamental group of a closed connected surface. If the surface is orientable, then its fundamental group is called an \emph{\textbf{orientable surface group}}, and a presentation is given by
\[\langle x_1,y_1,\dots,x_g,y_g\mid [x_1,y_1]\dots[x_g,y_g]=1\rangle,\]
for some $g\geq 1$ called the genus of the surface (if $g=0$, then the surface is a sphere, and its fundamental group is trivial).
If the surface is non-orientable, we call its fundamental group a \emph{\textbf{non-orientable surface group}}. It has a presentation given by
\[\langle x_1,\dots,x_g\mid x_1^2\dots x_g^2=1\rangle,\]
for some $g\geq 1$ called the genus of the surface.
A surface group is non-amenable if and only if it is the fundamental group of a surface other that the sphere, the torus (orientable surfaces of genus $0$ and $1$), the projective plane or the Klein bottle (non-orientable surfaces of genus $1$ and $2$).

\begin{df} Let $p$ be a prime number. A group $\Gamma$ is a \emph{\textbf{residually finite $p$-group}} if for every non-trivial element $\gamma\in\Gamma$, there exists a normal subgroup $N\trianglelefteq\Gamma$ such that $\Gamma/N$ is a finite $p$-group and $\gamma\notin N$. Equivalently, $\Gamma$ is a residually finite $p$-group if and only if there exists a chain $(\Gamma_n)_{n\geq 0}$ in $\Gamma$ consisting of normal subgroups such that for every $n\geq 0$, the quotient $\Gamma/\Gamma_n$ is a finite $p$-group, and \[\bigcap_{n\geq 0}\Gamma_n=\{1\}.\]
\end{df}




Baumslag proved in \cite{Baumslag} that orientable surface groups are residually free, i.e., for every non-trivial element $\gamma$, there exists a normal subgroup $N\trianglelefteq \Gamma$ such that $\Gamma/N$ is a free group and $\gamma\notin N$. Moreover, free groups are residually finite $p$-groups for every prime $p$, a result independently proved by Takahasi \cite{Takahasi} and by Gruenberg in \cite{Gruenberg} (using a result of Magnus \cite{Magnus}). This implies the following well-known result.

\begin{thm}\label{thm.residuellementp} Orientable surface groups are residually finite $p$-groups for every prime $p$.
\end{thm}

\begin{rmq}
By a result of Baumslag \cite{BaumslagBenjamin}, non-amenable non-orientable surface groups are also residually $p$-finite groups for every prime $p$. However, we leave as an exercise to the interested reader the fact that the fundamental group of a Klein bottle is not residually $p$ for some prime $p$. We will not require these results.
\end{rmq}

\subsection{Special kind of finite index subgroups in surface groups}

Let $A,B$ be two non-empty totally ordered finite sets. In what follows, when writing $\prod_{i\in A}$ or $\prod_{j\in B}$ we mean that the product is computed with respect to the increasing order of $A$ or $B$ respectively. We let $\Gamma_{A,B}$ be the group defined by the generators $(a_i,\alpha_i)_{i\in A}$ and $(b_j,\beta_j)_{j\in B}$, and the relation
\[\prod_{i\in A}[a_i,\alpha_i]=\prod_{j\in B}[b_j,\beta_j].\]
Then $\Gamma_{A,B}$ is isomorphic to a non-amenable orientable surface group, and every non-amenable orientable surface group is isomorphism to $\Gamma_{A,B}$ for some non-empty totally ordered finite sets $A$ and $B$. The group $\Gamma_{A,B}$ naturally splits as an amalgamated product
\[\Gamma_{A,B}=\Gamma_A*_\Z\Gamma_B\]
where $\Gamma_A$ and $\Gamma_B$ are the free groups of rank $2\lvert A\rvert$ and $2\lvert B\rvert$ respectively, freely generated by $(a_i,\alpha_i)_{i\in A}$ and $(b_j,\beta_j)_{j\in B}$ respectively.
If $A'\subset A$ and $B'\subset B$, there is a natural onto group homomorphism $\Gamma_{A,B}\twoheadrightarrow\Gamma_{A',B'}$ defined on the generators by
\[\begin{array}{llll}
a_i \mapsto a_i' &\text{ for every }i\in A',&\qquad\qquad b_j\mapsto b_j' &\text{ for every }j\in B',\\
\alpha_i \mapsto \alpha_i' &\text{ for every }i\in A',&\qquad\qquad \beta_j\mapsto \beta_j' &\text{ for every }j\in B',\\
a_i,\alpha_i \mapsto 1 &\text{ for every }i\in A\setminus A', &\qquad\qquad b_j,\beta_j\mapsto 1 &\text{ for every }j\in B\setminus B'.
\end{array}\]
We say that this morphism \emph{erases} the generators $a_i,\alpha_i,b_j,\beta_j$ for $i\in A\setminus A'$ and $j\in B\setminus B'$, see Figure \ref{fig.erasing}. Algebraically, $\Gamma_{A',B'}$ is isomorphic to the quotient of $\Gamma_{A,B}$ by the normal closure of the set $\{(\alpha_i,\beta_i)\mid i\in A\setminus A'\}\cup\{(b_j,\beta_j)\mid j\in B\setminus B'\}$ in $\Gamma_{A,B}$, and the homomorphism $\Gamma_{A,B}\onto\Gamma_{A',B'}$ corresponds to the quotient group homomorphism.

\begin{figure}[h!]
\centering
\def\svgwidth{\columnwidth}
\includegraphics[scale=1]{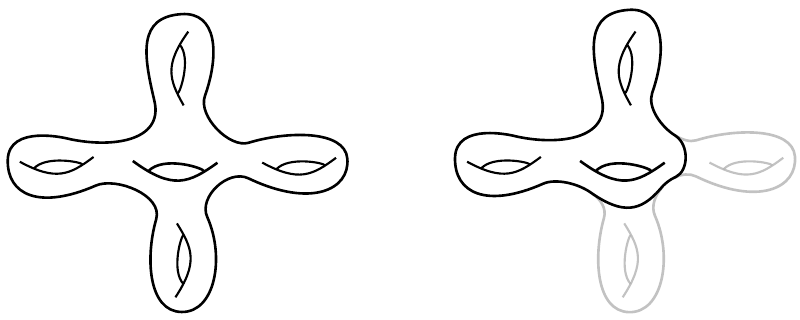}
\caption{An illustration of the morphism that erases generators.}
\label{fig.erasing}\end{figure}

Here is the main theorem of this section. In what follows, $\Z[1/p]$ denotes the set of rational numbers of the form $k/p^n$ for $k,n\in \Z$.

\begin{thm}\label{thm.sousgroupe}
\label{prop.sousgroupe} Let $\Gamma$ be a non-amenable orientable surface group, and fix a decomposition $\Gamma=\Gamma_A*_\Z\Gamma_B$ as above. Let $p$ be a prime number, and $r\in]0,1[\cap\Z[1/p]$. Let $\langle\langle \Z\rangle\rangle^{\Gamma_B}$ be the normal closure of the amalgamated subgroup $\Z$ in $\Gamma_B$. For every non-trivial $\gamma\in\Gamma$ and for every element $\delta\in\Gamma_B\setminus \langle\langle \Z\rangle\rangle^{\Gamma_B}$, there exists a finite index subgroup $\Lambda\leq\Gamma$ such that
\begin{itemize}
\item[$(i)$] $\gamma\notin\Lambda$.
\item[$(ii)$] The index $[\Gamma:\Lambda]$ is a power of $p$.
\item[$(iii)$] The number of left cosets $x\in\Gamma/\Lambda$ that are fixed by every element in $\Gamma_A$ is equal to $r[\Gamma:\Lambda]$.
\item[$(vi)$] None of the left coset $x\in\Gamma/\Lambda$ is fixed by $\delta$.
\end{itemize}
\end{thm}

\begin{proof}
Fix $A,B$ two non-empty totally ordered finite sets, such that $\Gamma$ is isomorphic to $\Gamma_{A,B}$. Let $S$ be the set of generators $(a_i,\alpha_i)_{i\in A}$ and $(b_j,\beta_j)_{j\in B}$. Let $j_0$ be the smallest element in $B$. Let $\gamma\in\Gamma\setminus\{1\}$ and $\delta\in\Gamma_B\setminus\langle\langle\Z\rangle\rangle^{\Gamma_B}$. Let $p$ be a prime number, and $r\in]0,1[\cap\Z[1/p]$.
\paragraph{Step 1: Cyclic covering.} Let $\varphi : \Gamma_{A,B}\onto\Z$ be the onto homomorphism defined on the generators of $\Gamma_{A,B}$ by
\[\begin{array}{l}
\varphi(b_{j_0})=1,\ \varphi(\beta_{j_0})=0,\\
\varphi(a_i)=\varphi(\alpha_i)=\varphi(b_j)=\varphi(\beta_j)=0 \text{ for every }i\in A, j\in B\setminus\{j_0\}.
\end{array}\]
For every $d\geq 1$, we let $\Lambda_d$ be the kernel of the homomorphism $\Gamma\to\Z/d\Z$ obtained by composing $\varphi$ with the homomorphism of reduction modulo $d$. Then $\Lambda_d$ is a surface group. Let us describe a generating set for $\Lambda_d$. For every $0\leq k\leq d-1$ and $i\in A$, let $a_{i,k}$ and $\alpha_{i,k}$ be the conjugates of $a_i$ and $\alpha_i$ respectively, by $b_{j_0}^k$. Similarly, let $b_{j,k}$, and $\beta_{j,k}$ be the conjugate of $b_j$ and $\beta_j$ respectively, by $b_{j_0}^k$. Then $\Lambda_d$ is generated by the set
\[\bigcup_{k=0}^{d-1}\{a_{i,k},\alpha_{i,k}\mid i\in A\}\cup \bigcup_{k=0}^{d-1}\{b_{j,k},\beta_{j,k}\mid j\in B\setminus\{j_0\}\}\cup \{b_{j_0}^{d},\beta_{j_0}\}.\]
So far, every left coset $x\in\Gamma/\Lambda_d$ is fixed by every element of $\Gamma_A$, and either every or none of the left coset $x\in\Gamma/\Lambda_d$ is fixed by $\delta$, depending on whether $\delta\in\Lambda_d$ or not.

\paragraph{Step 2: Erasing the right amount of generators.} Let $n$ be the length of $\gamma\in\Gamma\setminus\{1\}$ in the generating set $S$. In the sequel we let $d$ be a (large enough) power of the prime $p$ such that $rd$ is an integer, and $rd+n\leq d$. Let $E\subset \{n+1,\dots,d-1-n\}$ be a subset of cardinality $rd$, so that $\gamma$ doesn't belong to the normal closure $N$ of the set $\cup_{k\in E}b_{j_0}^k\Gamma_Ab_{j_0}^{-k}$ in $\Lambda_d$. Let us prove that none of the conjugate of $\delta$ by a power of $b_{j_0}$ belongs to $N$. Assume this is not the case, then this would imply that $\delta$ belongs to the normal closure of $\cup_{k=0}^{d-1}b_{j_0}^k\Gamma_Ab_{j_0}^{-k}$ in $\Lambda_d$, which is easily seen to be equal to the normal closure $\langle\langle \Gamma_A\rangle\rangle^\Gamma$ of $\Gamma_A$ in $\Gamma$. But the group $\Gamma/\langle\langle\Gamma_A\rangle\rangle^\Gamma$ is naturally isomorphic to $\Gamma_B/\langle\langle\Z\rangle\rangle^{\Gamma_B}$, in such a way that the following diagram commutes
\[\xymatrix{
 \Gamma_B \ar[d] \ar[rd] & \\
 \Gamma/\langle\langle\Gamma_A\rangle\rangle^\Gamma \ar[r] & \Gamma_B/\langle\langle\Z\rangle\rangle^{\Gamma_B},}
\]
which implies that $\Gamma_B\cap\langle\langle\Gamma_A\rangle\rangle^\Gamma$ is equal to $\langle\langle \Z\rangle\rangle^{\Gamma_B}$. This would thus imply that $\delta\in \langle\langle \Z\rangle\rangle^{\Gamma_B}$, a contradiction.


\paragraph{Step 3: The group $\Lambda_d/N$ is a residually finite $p$-group.} We let $\pi : \Lambda_d\onto \Lambda_d/N$ be the quotient group homomorphism. Since $\Lambda_d/N$ is an orientable surface group, it is a residually finite $p$-group by Theorem \ref{thm.residuellementp}. Thus, there exists a normal subgroup $N'\trianglelefteq \Lambda_d/N$ whose index is a power of $p$, such that for every $k\in\{0,\dots,d-1\}\setminus E$, for every $i\in A$, $\pi(a_{i,k})\notin N'$ and $\pi(\alpha_{i,k})\notin N'$. If $\gamma\in\Lambda_d$, we also assume that $\pi(\gamma)\notin N$, and if $\delta\in\Lambda_d$, we also assume that for all $k\in\{0,\dots,d-1\}$, $\pi(b_{j_0}^k\delta b_{j_0}^{-k})\notin N'$. Let us prove that the subgroup $\Lambda:=\pi^{-1}(N')$ of $\Gamma$ satisfies the four conclusions of the theorem.

\begin{proof}[Proof of $(i)$]\renewcommand{\qedsymbol}{} Either $\gamma\notin\Lambda_d$ and thus $\gamma\notin\Lambda$, or $\gamma\in\Lambda_d$ and $\pi(\gamma)\notin N$.\qedhere
\end{proof}
\begin{proof}[Proof of $(ii)$]\renewcommand{\qedsymbol}{} Since the index of $N'$ in $\Lambda_d/N$ is a power of $p$, $[\Lambda_d:\Lambda]$ is also a power of $p$. Thus $[\Gamma:\Lambda]=[\Gamma:\Lambda_d][\Lambda_d:\Lambda]$ is a power of $p$.
\end{proof}
\begin{proof}[Proof of $(iii)$]\renewcommand{\qedsymbol}{} By construction, $x\in\Gamma/\Lambda$ is fixed by every element in $\Gamma_A$ if and only its image under the canonical $[\Lambda_d:\Lambda]$-to-one map $\Gamma/\Lambda\mapsto\Gamma/\Lambda_d$ is equal to $b_{j_0}^k\Lambda_d$ for some $k\in E$. Since $\lvert E\rvert=rd$, there are exactly $rd[\Lambda_d:\Lambda]=r[\Gamma:\Lambda]$ such $x\in\Gamma/\Lambda$.
\end{proof}
\begin{proof}[Proof of $(iv)$] If $\delta\notin\Lambda_d$, then none of the coset $x\in\Gamma/\Lambda$ is fixed by $\delta$. If $\delta\in\Lambda_d$, then for all $k\in\{0,\dots,d-1\}$, we have $\pi(b_{j_0}^k\delta b_{j_0}^{-k})\notin N'$, and thus $\delta b_{j_0}^{-k}\Lambda\neq b_{j_0}^{-k}\Lambda$. By normality of $\Lambda$ in $\Lambda_d$, we deduce that none of the coset $x\in\Gamma/\Lambda$ is fixed by $\delta$.
\end{proof}
\renewcommand{\qedsymbol}{}
\end{proof}

\begin{figure}[h!]
\centering
\def\svgwidth{\columnwidth}
\includegraphics[scale=1]{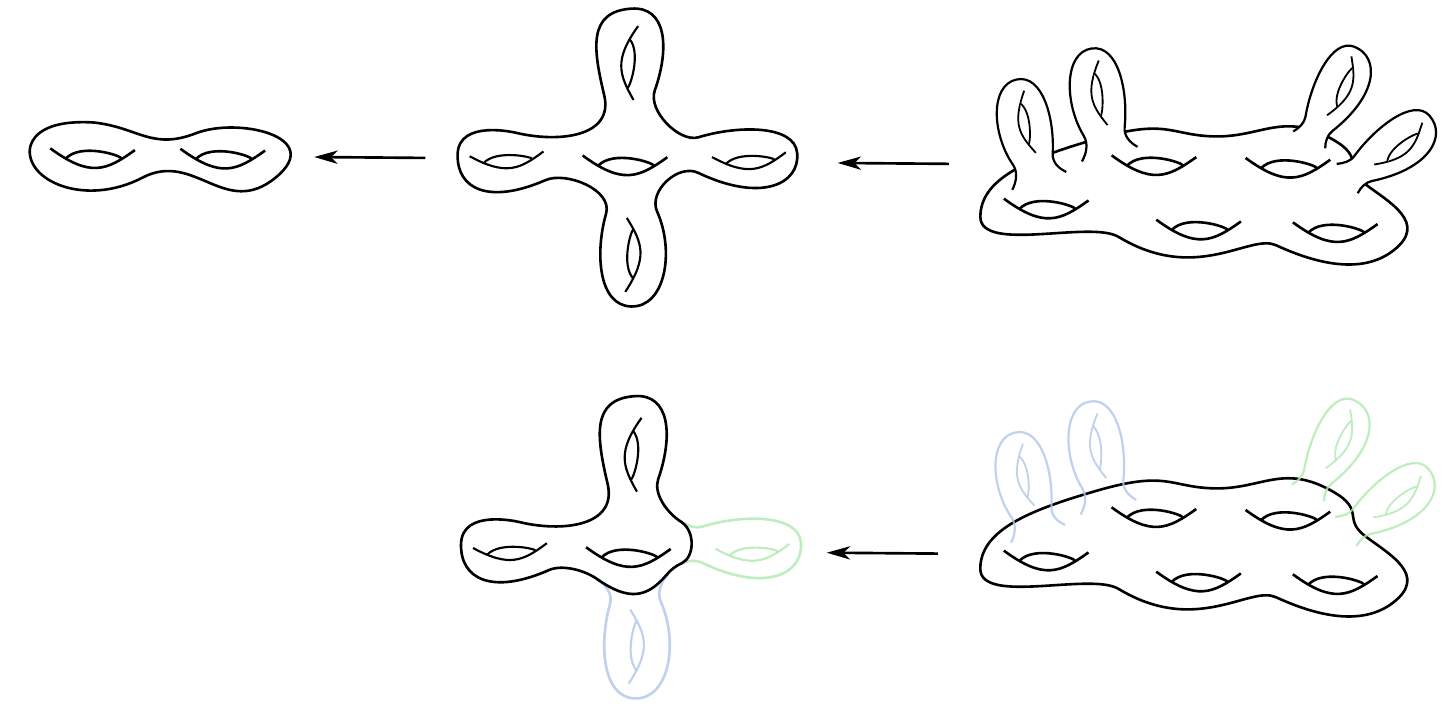}
\caption{Illustrations of the proof of Theorem \ref{thm.sousgroupe}. The above line illustrates the coverings corresponding to the inclusions $\Lambda\leq\Lambda_d\leq\Gamma$. The bottom line illustrates the covering corresponding to the inclusion $N'\leq \Lambda_d/N$.}
\label{fig.cycliccovering}\end{figure}

\section{Proof of the main theorem}\label{sec.proof}

In this section, we give the proof of Theorem \ref{thm.main}. More precisely, we prove the following results.

\begin{thm}[Orientable case]\label{thm.mainraffinerorientable} Let $\Gamma$ be a non-amenable orientable surface group, and fix a decomposition $\Gamma=\Gamma_A*_\Z\Gamma_B$ as above. Let $\langle\langle\Z\rangle\rangle^{\Gamma_B}$ be the normal closure of the amalgamated subgroup $\Z$ in $\Gamma_B$. Then there exists a continuum $(\alpha^t)_{0<t<1}$ of ergodic profinite allosteric actions of $\Gamma$ such that for all $0<t<1$,
\begin{enumerate}
\item The set of points whose stabilizer for $\alpha^t$ contains $\Gamma_A$ has measure $t$.
\item Each element of $\Gamma_B\setminus\langle\langle \Z\rangle\rangle^{\Gamma_B}$ acts essentially freely for $\alpha^t$.
\end{enumerate}
In particular, for all $0<s<t<1$, the actions $\alpha^s$ and $\alpha^t$ are neither topologically nor measurably isomorphic, and the probability measures $\IRS(\alpha^s)$ and $\IRS(\alpha^t)$ are distinct.
\end{thm}

\begin{thm}[Non-orientable case]\label{thm.mainraffinernonorientable} Let $\Gamma'$ be a non-amenable non-orientable surface group. Then there exists an index two subgroup $\Gamma\leq\Gamma'$ which is isomorphic to an orientable surface group, and which decomposes as $\Gamma=\Gamma_A*_\Z\Gamma_B$, and a continuum $(\beta^t)_{0<t<1}$ of ergodic profinite allosteric actions of $\Gamma'$ such that for all $0<t<1$, the set of points whose stabilizer for $\beta^t$ contains $\Gamma_A$ has measure $t/2$. In particular, for all $0<s<t<1$, the actions $\beta^s$ and $\beta^t$ are neither topologically nor measurably isomorphic, and the probability measures $\IRS(\beta^s)$ and $\IRS(\beta^t)$ are distinct.
\end{thm}

During the proof of these theorems, we will need the following lemma.

\begin{lem}\label{lem.chinois} Let $\Gamma$ be a group, and $\Lambda_1,\dots,\Lambda_n$ be finite index subgroups of $\Gamma$. If the indices $[\Gamma:\Lambda_i]$, $i\in\{1,\dots,n\}$, are pairwise coprime integers, then the left coset action $\Gamma\curvearrowright \Gamma/(\Lambda_1\cap\dots\cap\Lambda_n)$ is isomorphic to the diagonal action $\Gamma\curvearrowright \Gamma/\Lambda_1\times\dots\times\Gamma/\Lambda_n$ of the left coset actions $\Gamma\curvearrowright \Gamma/\Lambda_i$.\end{lem}

\begin{proof} The kernel of the group homomorphism $\Gamma\to \Gamma/\Lambda_1\times \dots\times\Gamma/\Lambda_n$ defined by $\gamma\mapsto (\gamma\Lambda_1,\dots,\gamma\Lambda_n)$ is equal to $\Lambda_1\cap\dots\cap\Lambda_n$. Thus $\Gamma/(\Lambda_1\cap\dots\cap\Lambda_n)$ is isomorphic to a subgroup of $\Gamma/\Lambda_1\times\dots\times\Gamma/\Lambda_n$. Moreover, for every $1\leq i\leq n$,
\[[\Gamma:\Lambda_1\cap\dots\cap \Lambda_n]=[\Gamma:\Lambda_i][\Lambda_i:\Lambda_1\cap\dots\cap\Lambda_n],\]
and since the indices $[\Gamma:\Lambda_i]$ are pairwise coprime, this implies that $[\Gamma:\Lambda_1\cap\dots\cap\Lambda_n]$ is divisible by $[\Gamma:\Lambda_1]\dots[\Gamma:\Lambda_n]$. Thus, the group homomorphism $\Gamma/\Lambda_1\cap\dots\cap\Lambda_n\to \Gamma/\Lambda_1\times\dots\times\Gamma/\Lambda_n$ is an isomorphism, and it is $\Gamma$-equivariant.
\end{proof}

We are now ready to prove Theorem \ref{thm.mainraffinerorientable} and Theorem \ref{thm.mainraffinernonorientable}.

\begin{proof}[Proof of Theorem \ref{thm.mainraffinerorientable}]

Let $\Gamma$ be a non-amenable orientable surface group, and we fix a decomposition $\Gamma=\Gamma_A*_\Z\Gamma_B$. Let $0<t<1$ be a real number. Let $(p_n)_{n\geq 1}$ be a sequence of pairwise distinct prime numbers. We fix a sequence $(r_n)_{n\geq 1}$ such that each $r_n$ belongs to $]0,1[\cap\Z[1/p_n]$ and $\prod_{n\geq 1}r_n=t$. Such a sequence exists because each $\Z[1/p_n]$ is dense in $\R$. Finally, let $(\gamma_n)_{n\geq 0}$ be an enumeration of the elements in $\Gamma$ with $\gamma_0=1$, and $(\delta_n)_{n\geq 1}$ be an enumeration of the elements in $\Gamma_B\setminus\langle\langle Z\rangle\rangle^{\Gamma_B}$. For every $n\geq 1$, there exists by Theorem \ref{prop.sousgroupe} a finite index subgroup $\Lambda_n^t\leq \Gamma$ which doesn't contain $\gamma_n$, whose index $[\Gamma:\Lambda_n^t]$ is a power of $p_n$, such that the number of left cosets $x\in\Gamma/\Lambda_n^t$ that are fixed by any element of $\Gamma_A$ is equal to $r_n[\Gamma:\Lambda_n^t]$, and such that none of the left coset $x\in\Gamma/\Lambda_n^t$ is fixed by $\delta_n$. For every $n\geq 1$, let $\Gamma_n^t:=\Lambda_1^t\cap\dots\cap\Lambda_n^t$. The sequence $(\Gamma_n^t)_{n\geq 1}$ forms a chain in $\Gamma$ and we denote by $\alpha^t$ the corresponding ergodic profinite action, and by $\mu_t$ the profinite $\Gamma$-invariant probability measure on $\varprojlim \Gamma/\Gamma_n^t$. This is a \pmp{} ergodic minimal action and we will prove that it is allosteric. By construction of $\Lambda_n^t$, we have that
\[\bigcap_{n\geq 1}\Gamma_n^t=\{1\}.\]
This implies by Lemma \ref{lem.URSprofinitetrivial} that $\URS(\alpha^t)$ is trivial. Let us prove that each element of $\Gamma_B\setminus\langle\langle\Z\rangle\rangle^{\Gamma_B}$ acts essentially freely for $\alpha^t$. Let $\delta\in\Gamma_B\setminus\langle\langle\Z\rangle\rangle^{\Gamma_B}$. By Lemma \ref{lem.chinois}, the number of $x\in\Gamma/\Gamma_n^t$ such that $\delta x=x$ is equal to the number of $(x_1,\dots,x_n)\in\Gamma/\Lambda_1^t\times\dots\times\Gamma/\Lambda_n^t$ such that $(\delta x_1,\dots,\delta x_n)=(x_1,\dots,x_n)$. If $n$ is large enough, then this last number is zero by construction of $\Lambda_n^t$. Thus, Lemma \ref{lem.closed} implies that $\Fix_{\alpha^t}(\delta)$ is $\mu_t$-negligible.

Finally, let us prove that the actions $\alpha^t$ are not essentially free. By construction, the indices $[\Gamma:\Lambda_i^t]$ are pairwise coprime. Thus, Lemma \ref{lem.chinois} implies that the number of $x\in \Gamma/\Gamma_n^t$ that are fixed by every element in $\Gamma_A$ is equal to the number of $(y_1,\dots,y_n)\in \Gamma/\Lambda_1^t\times\dots\times \Gamma/\Lambda_n^t$ that are fixed for the diagonal action by every element in $\Gamma_A$. By construction of $\Lambda_i^t$, this number is equal to $r_1[\Gamma:\Lambda_1^t]\times\dots\times r_n[\Gamma:\Lambda_n^t]$ which is equal to $r_1\dots r_n[\Gamma:\Gamma_n^t]$. Thus, Lemma \ref{lem.closed} implies that the $\mu_t$-measure of the set of points whose stabilizer for $\alpha^t$ contains $\Gamma_A$ is $t$. In particular, this implies that $\IRS(\alpha^t)$ is non-trivial. Thus $\alpha^t$ is allosteric. Moreover, this also implies that for all $0<s<t<1$, the actions $\alpha^s$ and $\alpha^t$ are not measurably isomorphic, and thus not topologically isomorphic since every $\alpha^t$ is uniquely ergodic by Lemma \ref{lem.ergodicminimal}, and this finally implies that the measures $\IRS(\alpha^s)$ and $\IRS(\alpha^t)$ are distinct.
 \end{proof}

\begin{proof}[Proof of Theorem \ref{thm.mainraffinernonorientable}]

Let $\Sigma'$ be a non-orientable surface of genus $g\geq 3$. Consider the usual embedding of an orientable surface $\Sigma$ of genus $g-1$ into $\R^3$ in such a way that the reflexions in all $3$ coordinate planes map the surface to itself, and let $\iota$ to be the fixed-point free antipodal map $x\mapsto -x$. Then $\Sigma'$ is homeomorphic to the quotient of $\Sigma$ by $\iota$, and the covering $\Sigma \mapsto \Sigma/\iota\approx\Sigma'$ is called the orientation covering. We decompose $\Sigma$ as the union of two surfaces $\Sigma_A$ and $\Sigma_B$ with one boundary, of genus $\lvert A\rvert$ and $\lvert B\rvert$ respectively, with $\lvert A\rvert\leq\lvert B\rvert$, so that $\iota(\Sigma_A)\subset\Sigma_B$. Fix a point $p\in\Sigma_A\cap\Sigma_B$, then Van Kampen's Theorem implies that the fundamental group $\Gamma$ of the surface $\Sigma$ based at $p$ is isomorphic to $\Gamma_A*_\Z\Gamma_B$ with $\Gamma_A=\pi_1(\Sigma_A,p), \Gamma_B=\pi_1(\Sigma_B,p)$ and $\Z\approx \pi_1(\Gamma_A\cap\Gamma_B,p)$. The fundamental group $\Gamma'$ of $\Sigma'$ based at $p'=\iota(p)$ naturally contains the subgroup $\Gamma$ as an index-two subgroup. Fix a curve contained in $\Sigma_B$ that joins $p$ to $\iota(p)$. This produces an element $\gamma_0\in\Gamma'\setminus\Gamma$, that satisfies $\gamma_0\Gamma_A\gamma_0^{-1}\leq\Gamma_B$.

Let $(\alpha^t)_{0<t<1}$ be a continuum of allosteric $\Gamma$-actions on $(X_t,\mu_t)$ given by Theorem \ref{thm.mainraffinerorientable}. The actions $\beta^t: \Gamma'\curvearrowright (Y_t,\nu_t)$ induced by the $\Gamma$-actions $\alpha^t$ are allosteric, see Proposition \ref{prop.allostericsousgroupe}. Let us prove that the set of points in $Y_t$ whose stabilizer for $\beta^t$ contains $\Gamma_A$ has $\nu_t$-measure $t/2$. Since $\beta^t$ is an induced action and $[\Gamma':\Gamma]=2$, the $\Gamma'$-action $\beta^t$ is measurably isomorphic to a \pmp{} $\Gamma'$-action on $(X_t\times\{0,1\},\mu_t\times\text{ unif})$, still denoted by $\beta^t$, that satisfies the following two properties:
\begin{enumerate}
\item For every $\gamma\in\Gamma'\setminus\Gamma$, the sets $X_t\times \{0\}$ and $X_t\times \{1\}$ are switched by $\beta^t(\gamma)$.
\item For every $\gamma\in\Gamma$, for every $x\in X_t$, $\beta^t(\gamma)(x,0)=(\alpha^t(\gamma)x,0)$ and $\beta^t(\gamma)(x,1)=(\alpha^t(\gamma_0\gamma\gamma_0^{-1})x,1)$.
\end{enumerate}
This implies that for all $(x,\varepsilon)\in X_t\times\{0,1\}$, the subgroup $\Gamma_A$ is contained in $\Stab_{\beta^t}(x,\varepsilon)$ if and only if either $\varepsilon=0$ and $\Gamma_A$ is contained in $\Stab_{\alpha^t}(x)$, or $\varepsilon=1$ and $\gamma_0\Gamma_A\gamma_0^{-1}$ is contained in $\Stab_{\alpha^t}(x)$. Thus, the set of points whose stabilizer for $\beta^t$ contains $\Gamma_A$ has $\nu_t$-measure
\[\frac{t+\mu_t(\{x\in X_t\mid \gamma_0\Gamma_A\gamma_0^{-1}\leq\Stab_{\alpha^t}(x)\})}{2}.\]
In order to finish the proof, it is enough to prove that the intersection of $\gamma_0\Gamma_A\gamma_0^{-1}$ and $\Gamma_B\setminus\langle\langle\Z\rangle\rangle^{\Gamma_B}$ is non-trivial, since any element in $\Gamma_B\setminus\langle\langle\Z\rangle\rangle^{\Gamma_B}$ acts essentially freely for $\alpha^t$. The conjugation by $\gamma_0$ induces a group automorphism $\varphi:\Gamma\mapsto\Gamma$, such that $\varphi(\Gamma_A)\leq\Gamma_B$. Since $\Gamma_A$ is not contained in the derived subgroup $D(\Gamma)$, so is $\varphi(\Gamma_A)$. But the amalgamated subgroup $\Z$ is contained in $D(\Gamma)$, thus so is $\langle\langle \Z\rangle\rangle^{\Gamma_B}$. This implies that the intersection $\varphi(\Gamma_A)\cap(\Gamma_B\setminus\langle\langle\Z\rangle\rangle^{\Gamma_B})$ is non-empty. We deduce that the set of points whose stabilizer for $\beta^t$ contains $\Gamma_A$ has $\nu_t$-measure $t/2$. We conclude that the actions $\beta^t$ are neither measurably nor topologically pairwise isomorphic and that their IRS are pairwise disjoint as in Theorem \ref{thm.mainraffinerorientable}.
\end{proof}

\begin{rmq}\label{rmq.totipotent} Let $\alpha : \Gamma\curvearrowright(C,\mu)$ be an allosteric action. Then we have 
\[\supp(\IRS(\alpha))\subset \overline{\{\Stab_{\alpha}(x)\mid x\in C\}}.\]
This implies that the support of $\IRS(\alpha)$ doesn't contain any non-trivial subgroup with only finitely many conjugates, because otherwise the closure of the set $\{\Stab_\alpha(x)\mid x\in C\}$ would contain a closed minimal $\Gamma$-invariant set $\neq\{\{1\}\}$. Carderi, Gaboriau and Le Ma\^itre proved (personal communication) that the perfect kernel of a surface group coincides with the set of its infinite index subgroups. This implies that allosteric actions of surface groups are not totipotent (a \pmp{} action is \emph{\textbf{totipotent}} if the support of its IRS coincide with the perfect kernel of the group, see \cite{CarderiGaboriauLeMaitre}).\end{rmq}

\begin{rmq}\label{rmq.weaklymixing} A \pmp{} action $\Gamma\curvearrowright(X,\mu)$ is \emph{\textbf{weakly mixing}} if for every $\varepsilon >0$ and every finite collection $\Omega$ of measurable subsets of $X$, there exists a $\gamma\in\Gamma$ such that for every $A,B\in\Omega$ \[\lvert \mu(\gamma A\cap B)-\mu(A)\mu(B)\rvert <\varepsilon.\]
With this definition, it is easily seen that the restriction of a weakly mixing action to a finite index subgroup remains weakly mixing. Thus Proposition \ref{prop.mixing} implies that the IRS's of non-amenable surface groups we have constructed are not weakly mixing.
\end{rmq}

\begin{rmq} The proof of our main theorem applies mutatis mutandis to branched orientable surface groups, that is fundamental groups of closed orientable branched surfaces (see Figure \ref{fig.surfacebranchee}). These groups can be written as amalgams. Fix an integer $g\geq 2$ as well as $2g$ letters $x_1,y_1,\dots,x_g,y_g$. Fix a partition of $\{1,\dots,g\}$ into $n$ nonempty intervals $A_1,\dots,A_n$. Let $\Gamma_k$ be the free group generated by $x_i$ and $y_i$ for every $i\in A_k$, and let $\Z\to \Gamma_k$ be the injective homomorphism defined by sending the generator of $\Z$ to the product $\prod_{i\in A_k}[x_i,y_i]$. Then the amalgam $*_\Z\Gamma_i$
is a branched orientable surface group, and any branched orientable surface group can be obtained this way. The fundamental group of a closed orientable branched surface of genus $\geq 2$ is a residually $p$-finite group for every prime $p$, see \cite[Theorem $4.2$]{KimMccarron}. Thus our method of proof applies to branched orientable surface groups, with any $\Gamma_k$ in the role played by $\Gamma_A$ during the proof of Theorem \ref{thm.mainraffinerorientable}.

\begin{figure}[h!]
\centering
\def\svgwidth{\columnwidth}
\scalebox{0.6}{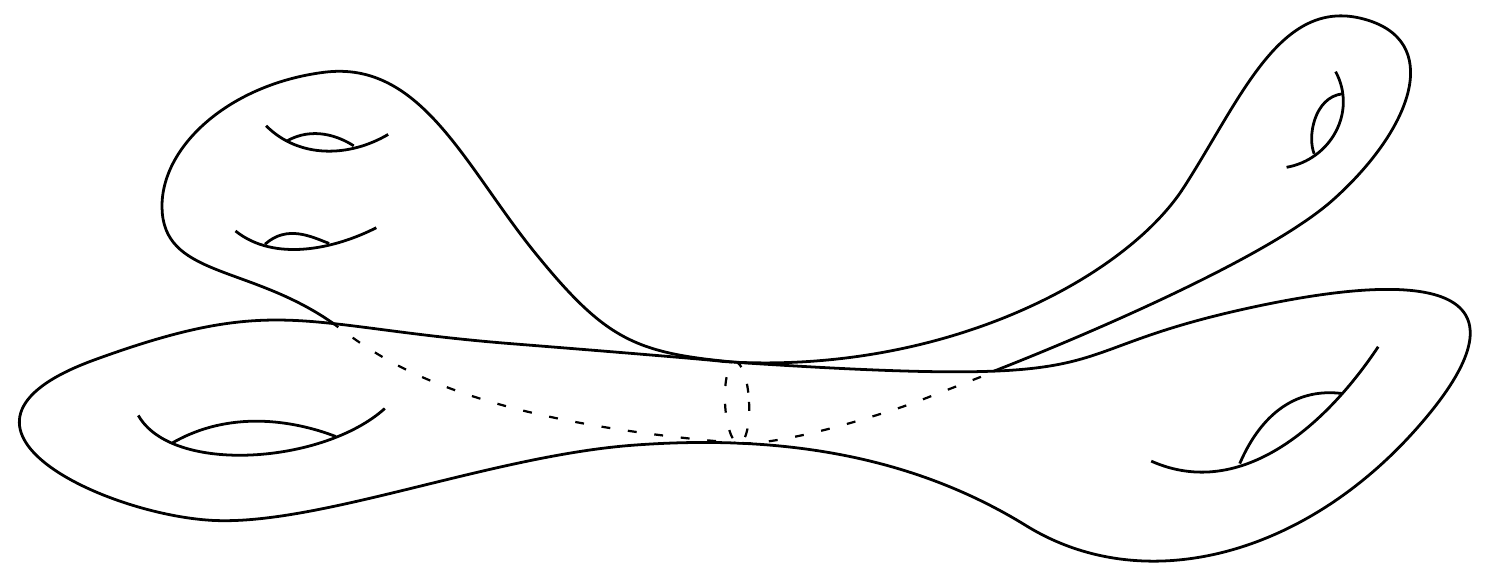}
\caption{A branched surface}
\label{fig.surfacebranchee}\end{figure}
\end{rmq}

\begin{qu} Is the fundamental group of a compact hyperbolic $3$-manifold allosteric? More generally, is the fundamental group of a compact orientable aspherical $3$-manifold allosteric?
\end{qu}

	\bibliographystyle{alpha}
	\nocite{*}
	\bibliography{biblio}

	\Addresses
\end{document}